\documentclass[fleqn,preprint,3p,a4paper]{elsarticle}
\usepackage{amssymb}
\usepackage{amsmath}
\usepackage{amsthm}
\usepackage{dcolumn}
\usepackage{endnotes}
\usepackage{tabularx}
\usepackage[matrix,arrow]{xy}
\usepackage{wasysym}
\usepackage{graphicx}

\newtheorem{theorem}{Theorem}[section]
\newtheorem{proposition}[theorem]{Proposition}
\newtheorem{lemma}[theorem]{Lemma}
\newtheorem{corollary}[theorem]{Corollary}

\theoremstyle{definition}

\newtheorem{remark}[theorem]{Remark}

\newcommand{\ir}{{\mathsf{Irr}}}

\newcommand{\mn}{\mathbb N}

\newcommand{\ua}{{\uparrow}}
\newcommand{\da}{{\downarrow}}

\newcommand{\N}{\mathbb{N}}
\newcommand{\B}{\mathbb{B}}
\newcommand{\cQ}{\mathcal{Q}}
\newcommand{\cS}{\mathcal{S}}
\newcommand{\RO}{\operatorname{RO}}
\newcommand{\cl}{\operatorname{cl}}
\newcommand{\Int}{\operatorname{int}}
\newcommand{\Min}{\operatorname{min}}
\newcommand{\Max}{\operatorname{max}}

\begin{document}
\begin{frontmatter}

\title{Scott spaces of complete Boolean algebras need not be co-sober\tnoteref{t1}}
\tnotetext[t1]{This research was supported by the National Natural Science Foundation of China (Nos. 12561083, 12471070, 12071199).}

\author[W. Ji]{Wei Ji}
\ead{javeey@163.com}
\address[W. Ji]{School of Mathematics and Statistics, Guilin University of Technology,
  Guilin 541004, China}

\author[X. Xu]{Xiaoquan Xu\corref{mycorrespondingauthor}}
\cortext[mycorrespondingauthor]{Corresponding author}
\ead{xiqxu2002@163.com}
\address[X. Xu]{School of Mathematics and Statistics,
Minnan Normal University, Zhangzhou 363000, China}

%\nonumnote{* This is the corresponding author.}

\begin{abstract}
In this paper, we first prove that for a complete Boolean algebra $L$, the complement graph of $L$ is a \emph{KC}-space (as a subspace) and each
non-singleton compact irreducible subspace of that graph generates a
non-principal $k$-irreducible compact saturated set in the Scott space of the
square algebra $L\times L$. We then show that every compact sequential \emph{US}-space embeds
into the complement graph of a suitable complete Boolean algebra. The
embedding is built from a finite tail-constraint poset and its regular-open
completion. Applying the construction to van Douwen's compact
Fr\'echet anti-Hausdorff \emph{US}-space gives a complete Boolean algebra $B$ whose Scott
space $\Sigma~\!\!B$ is not co-sober, thereby answering negatively a question on Scott
spaces of complete Boolean algebras. The same Scott space  $\Sigma~\!\!B$ is non-sober and, as a Scott space of a complete lattice,
is well-filtered.
\end{abstract}


\end{frontmatter}

\section{Introduction}

In his pioneering work in what has come to be called ``domain theory" which provides a
mathematical foundation for the denotational semantics of programming languages, Dana Scott
introduced a crucial $T_0$-topology which came to be called the Scott topology. In domain theory
and non-Hausdorff topology, sobriety is probably the
most important and useful property of $T_0$-spaces (see \cite{Abramsky-Jung-1994, GHKLMS-2003, Goubault-2013, Xu-Zhao-2021}). The Hofmann-Mislove Theorem reveals a very distinct characterization for the sober spaces via open filters and illustrates the close relationship between domain theory and topology.

Co-sobriety was introduced by Escard\'o, Lawson and Simpson in connection
with a dual form of the Hofmann--Mislove theorem
\cite{Escardo-Lawson-Simpson-2004}. A $T_0$-space is co-sober when each
$k$-irreducible compact saturated set is the saturation of a point. Despite
the terminology, sobriety and co-sobriety are not formal duals. Co-sober spaces may not be sober (for example, the Scott space of Johnstone dcpo \cite{Johnstone-1981}. \cite[Example 2.1]{Wen-Xu-2018} is another example), and sober spaces
need not be co-sober (see, e.g., \cite{Wen-Xu-2018, Shen-Wu-Xi-Zhao-2020}, and several basic
preservation questions remain unsettled (see \cite{Xu-Zhao-2021, Zhao-He-Wang-2023}).

Two recent developments provide the immediate context for the present work.
Xi, Shen and Zhao proved that the Scott space of the complete Boolean algebra
of regular open subsets of the real line is not sober
\cite{Xi-Shen-Zhao-2026}. Xu, Xu, Ji and Qiu showed that the Smyth power space of
a co-sober space need not be co-sober \cite{Xu-Xu-Ji-Qiu-2026}. In an invited
talk at ISDT 2026, Xu asked whether the Scott space of every complete
Boolean algebra is co-sober \cite{Xu-2026}. The answer is negative.

The proof is organized around a complement graph. For a complete Boolean
algebra $B$, let

$$  G_B=\{(b,\neg b):b\in B\}\subseteq B^2.$$

\noindent We first show that $G_B$, with the topology inherited from the Scott space of
$B^2$, is a \emph{KC}-space. If $M$ is a non-singleton compact irreducible subspace
of $G_B$, then ${\uparrow} M$ is a compact saturated subset of $\Sigma~\!\!B^2$.
A Scott-closed disjointness locus recovers $M$ from ${\uparrow} M$ and implies
that ${\uparrow} M$ is $k$-irreducible. Its set of minimal elements is exactly
$M$, so it cannot be principal.

It remains to realize a suitable compact irreducible space inside a
complement graph. We prove that every compact sequential \emph{US}-space embeds
in $G_B$ for some complete Boolean algebra $B$. The algebra is the
regular-open completion of a poset whose conditions impose finitely many
eventual relations on convergent sequences together with a finite binary
coloring. A tail-density lemma converts topological convergence in the
original space into order convergence in $B$, and hence into Scott
convergence in $G_B$. Van Douwen's compact Fr\'echet
anti-Hausdorff \emph{US}-space then supplies the required compact irreducible
subspace \cite{Douwen-1993, Xu-Xu-Ji-Qiu-2026}.

The resulting Scott space is also non-sober. Thus the construction yields a
well-filtered Scott space of a complete Boolean algebra that is neither sober
nor co-sober. Section~\ref{sec:prelim} fixes the terminology.
Section~\ref{sec:graph} develops the complement-graph reduction.
Section~\ref{sec:realization} proves the Boolean realization theorem, and
Section~\ref{sec:main} gives the counterexample and its consequences.

\section{Preliminaries}\label{sec:prelim}

We use standard terminology from general topology and domain theory; see
\cite{Engelking-1989, GHKLMS-2003, Goubault-2013}.  All compactness
statements are made without a separation assumption.

For a poset $P$ and $A\subseteq P$, write
 $ {\uparrow} A=\{x\in P:(\exists a\in A)\ a\leq x\}$ and ${\downarrow} A=\{x\in P:(\exists a\in A)\ x\leq a\}$.
For $x\in P$, let ${\ua} x={\ua}\{x\}$ and ${\da} x={\da}\{x\}$.  A
nonempty set $D\subseteq P$ is directed if every finite subset of $D$ has an
upper bound in $D$.  Dually, $E\subseteq P$ is filtered if every finite
subset has a lower bound in $E$.

For a $T_0$-space $X$, we use $\leq_X$ to denote the \emph{specialization order} of $X$: $x\leq_X y$ if{}f $x\in \cl{\{y\}}$. A nonempty subset $A$ of $X$ is said to be \emph{irreducible} if for any $F_1$, $F_2\in\Gamma(X)$, $A\subseteq F_1\cup F_2$ implies $A\subseteq F_1$ or $A\subseteq F_2$. The set of all irreducible closed subsets of $X$ is denoted by $\ir_c(X)$. The space $X$ is called \emph{sober}, if for any $A\in\ir_c(X)$, there is a unique point $x\in X$ such that $A=\cl{\{x\}}$. A nonempty subspace $Y$ is irreducible if every two nonempty open subsets of $Y$ intersect. This is equivalent to the usual closed-set definition.  We also use the term \emph{anti-Hausdorff} for this property (see \cite{Hoffmann-1979-1}).

A subset $U$ of a poset $P$ is \emph{Scott open} if it is an upper set and, whenever a
directed set $D$ has a supremum with $\bigvee D\in U$, one has
$D\cap U\neq\varnothing$. The resulting topology is denoted by $\sigma(P)$,
and $\Sigma P=(P,\sigma(P))$ is the \emph{Scott space} of $P$. Its specialization
order is the original order. The \emph{dual Scott topology} on a complete
lattice $L$ means the Scott topology of the opposite order $L^{\mathrm{op}}$;
its closed sets are the upper sets closed under infima of filtered subsets.

A subset $A$ of a $T_0$-space $Y$ is called \emph{saturated} if $A$ equals the intersection of all open sets containing it (equivalently, $A$ is an upper set in the specialization order). We use $\cQ(Y)$ to denote the set of all nonempty compact saturated subsets of $Y$ and endow it with the Smyth order $\sqsubseteq: K_1\sqsubseteq K_2$ iff $K_{2}\subseteq K_{1}$. The space $Y$ is called \emph{well-filtered} if for any filtered family $\mathcal{K}\subseteq \cQ(Y)$ and any open set $U$, $\bigcap\mathcal{K}{\subseteq} U$ implies $K{\subseteq} U$ for some $K{\in}\mathcal{K}$. It is well-known that every sober space is well-filtered (see, e.g., \cite[Theorem II-1.21]{GHKLMS-2003}).

Let $X$ be a $T_0$-space. Denote by $\cQ(X)$ the family of nonempty compact
saturated subsets of $X$. An element $K\in\cQ(X)$ is \emph{$k$-irreducible}
if

$$  K=K_1\cup K_2,\quad K_1,K_2\in\cQ(X),$$

\noindent implies $K=K_1$ or $K=K_2$.  The space $X$ is \emph{co-sober} if every
$k$-irreducible member of $\cQ(X)$ has the form $\uparrow x$ for a point
$x\in X$ \cite{Escardo-Lawson-Simpson-2004}.  The representing point is unique
because $X$ is $T_0$.

As in \cite{Douwen-1993, Kunzi-Zypen-2004}, a topological space $X$ is called a \emph{KC}-\emph{space} if each compact set of $X$ is closed. The space $X$ is said to be a \emph{US}-\emph{space} provided that each convergent sequence has a unique limit. It is known that each Hausdorff space (=$T_2$-space) is a \emph{KC}-space, each \emph{KC}-space is a \emph{US}-space, each \emph{US}-space is a $T_1$-space and each first-countable \emph{US}-space is a $T_2$-space (cf. \cite{Wilansky-1967}).

A space $X$ is called a \emph{Fr\'echet space} if for every $A\subseteq X$ and $x\in \cl A$, there exists a sequence $x_1, x_2, ..., x_n, ...$ of $A$ converging to $x$. The space $X$ is called a \emph{sequential space} if a set $A\subseteq X$ is closed if and only if $\mathrm{lim}~\! a_n\subseteq A$ for any sequence $(a_n)_{n\in \mn}$ of $A$, where $\mathrm{lim}~\! a_n$ is the set of all limits of $(a_n)_{n\in \mn}$ in $X$. Every first-countable space is a Fr\'echet space and every Fr\'echet space is a sequential space (see, e.g., \cite[Theorem 1.6.14]{Engelking-1989}).

Throughout, $B$ denotes a complete Boolean algebra.  Its complement is
written $\neg b$.  Products of Boolean algebras carry the coordinatewise
order and operations.  We shall repeatedly use

$$  a\wedge\bigvee S=\bigvee_{s\in S}(a\wedge s).$$

\noindent This identity holds in every complete Boolean algebra because every complete Boolean algebra is a complete Heyting algebra, more precisely, the maps
$a\wedge(-)$ always have the right adjoints $\neg a\vee(-)$.

\section{Compact irreducible subspaces of complement graphs}
\label{sec:graph}

For a complete Boolean algebra $B$, define

$$  G_B=\{(b,\neg b):b\in B\}
  \quad\text{and}\quad
  \Delta_B=\{(a,c)\in B^2:a\wedge c=0\}.$$

\noindent The first set is an antichain: if
$(b,\neg b)\leq(c,\neg c)$, then $b\leq c$ and $\neg b\leq\neg c$, and hence
$c\leq b$ and $b=c$.

\begin{lemma}\label{lem:delta-closed}
For every complete Boolean algebra $B$, the set $\Delta_B$ is Scott closed in
$B^2$.
\end{lemma}

\begin{proof}
It is a lower set.  Let $D\subseteq\Delta_B$ be directed and put

$$  a=\bigvee_{(u,v)\in D}u,~ c=\bigvee_{(u,v)\in D}v.$$

\noindent For $(u_1,v_1),(u_2,v_2)\in D$, choose $(u_3,v_3)\in D$ above both.  Then
$u_1\wedge v_2\leq u_3\wedge v_3=0$.  Distributivity of finite meets over arbitrary joins gives

$$  a\wedge c
   =\bigvee_{(u_1,v_1),(u_2,v_2)\in D}(u_1\wedge v_2)=0.$$

\noindent Thus $\bigvee D=(a,c)$ belongs to $\Delta_B$.
\end{proof}

The next observation will also be used later.

\begin{lemma}\label{lem:compact-dual-closed}
For complete lattice $L$, every compact saturated subset of $\Sigma L$
is closed in the dual Scott topology.
\end{lemma}

\begin{proof}
Let $K$ be compact and saturated.  Then $K$ is an upper set.  It is enough to
show that $K$ is closed under infima of filtered subsets.  Let
$E\subseteq K$ be filtered and set $e=\bigwedge E$.  Suppose that
$e\notin K$.  For $d\in E$, the set $U_d=L\setminus\downarrow d$ is Scott open.  The family $\{U_d:d\in E\}$ covers $K$: otherwise some
$x\in K$ would satisfy $x\leq d$ for every $d\in E$, whence $x\leq e$ and,
since $K$ is upper, $e\in K$.

No finite subfamily of $\{U_d:d\in E\}$ covers $K$.  Indeed, for $d_1,\ldots,d_n\in E$, choose
$d\in E$ with $d\leq d_i$ for all $i$.  Then $d\in K$ and
$d\notin U_{d_i}$ for every $i$.  This contradicts compactness.  Therefore
$e\in K$.
\end{proof}

\begin{proposition}\label{prop:graph-kc}
For every complete Boolean algebra $B$, the complement graph $G_B$, with the
subspace topology inherited from  $\Sigma~\!\!B^2$ , is a KC-space.
\end{proposition}

\begin{proof}
Let $A\subseteq G_B$ be compact.  Its saturation ${\uparrow} A$ is compact in
 $\Sigma~\!\!B^2$ : every Scott-open cover of ${\uparrow} A$ restricts to one of
$A$, and a finite subcover of $A$ covers ${\uparrow} A$ because Scott-open sets
are upper.

By Lemma~\ref{lem:compact-dual-closed}, $\uparrow A$ is dual-Scott closed.
The map

$$  \kappa:B^2\longrightarrow B^2,
  ~ \kappa(x,y)=(\neg x,\neg y),$$

\noindent is an order anti-isomorphism.  Hence

$$  \kappa({\uparrow} A)={\downarrow}\kappa(A) $$

\noindent is Scott closed.  Since $G_B$ is an antichain,

$$  {\downarrow}\kappa(A)\cap G_B=\kappa(A).$$

\noindent Thus $\kappa(A)$ is closed in $G_B$.  On $G_B$, the map $\kappa$ agrees with
the coordinate swap $(x,y)\mapsto(y,x)$, which is a Scott homeomorphism of
$B^2$.  Therefore $A$ is closed in $G_B$.
\end{proof}

We can now isolate the topological object needed for a counterexample.

\begin{theorem}[Complement-graph reduction]\label{thm:reduction}
Let $B$ be a complete Boolean algebra.  If $G_B$ contains a non-singleton
compact irreducible subspace $M$, then $\Sigma~\!\!B^2$ is not co-sober.
\end{theorem}

\begin{proof}
Set $K=\uparrow M$. Then $K$ is compact and saturated.  We first note
that
\begin{equation}\label{eq:recover-M}
 ~~~~~~~~~~~~~~~~~~~~~~~~~~~~~~~~~~~~~~~~~~~~~~~~~~~~~K\cap\Delta_B=M.
\end{equation}
The inclusion $M\subseteq K\cap\Delta_B$ is immediate.  Conversely, let
$(a,c)\in K\cap\Delta_B$.  Then
$(b,\neg b)\leq(a,c)$ for some $(b,\neg b)\in M$.  Hence

$$  b\leq a,
  ~ \neg b\leq c,
  ~ a\leq\neg c.$$

\noindent The second inequality gives $\neg c\leq b$, so
$b\leq a\leq\neg c\leq b$.  Therefore $(a,c)=(b,\neg b)\in M$.

Suppose that $K=K_1\cup K_2$ with $K_1,K_2$ compact and saturated.  Put
$A_i=K_i\cap\Delta_B$.  By Lemma~\ref{lem:delta-closed}, each $A_i$ is
compact.  Equation~\eqref{eq:recover-M} shows that $A_i\subseteq M\subseteq
G_B$.  Proposition~\ref{prop:graph-kc} implies that $A_i$ is closed in
$G_B$, hence in $M$.  Moreover,

$$  M=(K_1\cap\Delta_B)\cup(K_2\cap\Delta_B)=A_1\cup A_2.$$

\noindent The irreducibility of $M$ yields $M=A_1$ or $M=A_2$.  If $M=A_1$, then the
upperness of $K_1$ gives
$K=\uparrow M\subseteq K_1\subseteq K$, so $K_1=K$.  The other case is the
same.  Thus $K$ is $k$-irreducible.

Finally,
\begin{equation}\label{eq:minK}
  ~~~~~~~~~~~~~~~~~~~~~~~~~~~~~~~~~~~~~~~~~~~~~~~~~~~~~~ \Min K=M.
\end{equation}
Indeed, if $m\in M$ and $x\in K$ satisfies $x\leq m$, choose $m'\in M$ with
$m'\leq x$.  Since $G_B$ is an antichain, $m'=m$ and $x=m$.  Conversely,
if $x$ is minimal in $K$, choose $m\in M$ with $m\leq x$; minimality gives
$x=m$.  Since $M$ is not a singleton, $K$ cannot be a principal upper set.
Therefore $\Sigma~\!\!B^2$ is not co-sober.
\end{proof}

The same subspace also detects non-sobriety.

\begin{proposition}\label{prop:nonsober-reduction}
Under the hypotheses of Theorem~\ref{thm:reduction}, the Scott space
 $\Sigma~\!\!B^2$  is not sober.
\end{proposition}

\begin{proof}
Let $F=\downarrow M$.  The set $\kappa(M)$ is compact in $G_B$, hence
$\uparrow\kappa(M)$ is compact saturated in  $\Sigma~\!\!B^2$ .  By
Lemma~\ref{lem:compact-dual-closed} it is dual-Scott closed, and applying
$\kappa$ shows that

$$  F=\kappa({\uparrow}\kappa(M))$$

\noindent is Scott closed.

If Scott-open sets $U$ and $V$ both meet $F$, their upperness implies that
each meets $M$.  The sets $U\cap M$ and $V\cap M$ are nonempty open subsets
of the irreducible space $M$, so $U\cap V\cap M\neq\varnothing$.  Hence $F$
is irreducible. We also have $\Max F=M$. Indeed, every $m\in M$ is
maximal in $F$ because $M$ is an antichain; conversely, if $x$ is maximal in
$F$, then $x\leq m$ for some $m\in M$, and maximality forces $x=m$.
Because $M$ is not a singleton, $F$ is not of the form ${\downarrow} x$, the
closure of a point in a Scott space. Thus $\Sigma~\!\!B^2$ is not sober.
\end{proof}

\section{Boolean realization of compact sequential \emph{US}-spaces}
\label{sec:realization}

We now construct a complete Boolean algebra whose complement graph contains
a prescribed compact sequential \emph{US}-space.  The construction uses only
finite constraints, but it imposes one eventual relation for each convergent
sequence.

\subsection{Finite tail constraints}

Let $X$ be a \emph{US}-space.  Write $\cS(X)$ for the set of convergent sequences in
$X$.  If $s=(s_n)_{n\in\N}\in\cS(X)$, its unique limit is denoted by
$s_\infty$.

The following result is well-known and can be easily verified (see, e.g., \cite{Wilansky-1967}).

\begin{lemma}\label{lem:US-T1}
Every US-space is $T_1$.
\end{lemma}

\begin{lemma}\label{lem:finite-overlap}
Let $s_n\to x$ and $t_n\to y$ in a US-space, where $x\neq y$.  Then the two
ranges $\{s_n:n\in\N\}$ and $\{t_n:n\in\N\}$ have finite intersection.
\end{lemma}

\begin{proof}
If the intersection contained infinitely many distinct points, choose
pairwise distinct common values whose indices in the sequence $s$ increase.
Their chosen indices in $t$ are distinct, and hence have an increasing
subsequence.  The corresponding common-value subsequence would converge to
both $x$ and $y$, contrary to the \emph{US} property.
\end{proof}

The next two elementary lemmas isolate the topological and relational
ingredients needed in the construction.

\begin{lemma}\label{lem:tail-block-avoidance}
Suppose that $s_n\to x$ and $t_n\to y$ in a US-space, where $x\neq y$.  If

 $$ x\notin R(t,N):=\{y\}\cup\{t_n:n\geq N\},$$

\noindent then the set $\{n:s_n\in R(t,N)\}$ is finite.
\end{lemma}

\begin{proof}
By Lemma~\ref{lem:finite-overlap}, the two ranges have only finitely many
common values.  The point $y$ adds at most one further value.  None of these
values is $x$.  Since the space is $T_1$ by
Lemma~\ref{lem:US-T1}, a sequence converging to $x$ takes each fixed value
$z\neq x$ only finitely often.  Hence $s$ enters $R(t,N)$ only finitely
often.

%By Lemma~\ref{lem:finite-overlap}, the set

%$$  A=\{s_k:k\in\N\}\cap\{t_n:n\geq N\}$$

%\noindent is finite.  The hypothesis gives $x\notin A\cup\{y\}$.  Since the space
%is $T_1$ by Lemma~\ref{lem:US-T1}, a sequence converging to $x$ takes each
%fixed value $z\neq x$ only finitely often.  The finite set
%$A\cup\{y\}$ contains every value assumed by $s$ while it lies in
%$R(t,N)$, so $s$ enters $R(t,N)$ only finitely often.
\end{proof}

\begin{lemma}[Isolation outside the old tail blocks]\label{lem:finite-component-structure}
Let

$$  \Gamma=\{(t^1,N_1),\ldots,(t^m,N_m)\}$$

\noindent be finite, and let $\sim_\Gamma$ be the equivalence relation generated by
$t^j_n\sim_\Gamma t^j_\infty$ for $n\geq N_j$.  For a
$\sim_\Gamma$-class $C$, put

$$  J_C=\{j:t^j_\infty\in C\}.$$

\noindent Then either $J_C=\varnothing$ and $C$ is a singleton, or

$$  C=\bigcup_{j\in J_C}
       \bigl(\{t^j_\infty\}\cup\{t^j_n:n\geq N_j\}\bigr).$$

\noindent In particular, every non-singleton class is a union of finitely many tail
blocks whose limits lie in that class.
\end{lemma}

\begin{proof}
View the generating relations as the edges of an undirected graph on $X$.
The equivalence classes are its connected components.  If a vertex of a
component is incident with an edge arising from $(t^j,N_j)$, then the center
$t^j_\infty$ belongs to the same component; conversely, every tail term of a
center in the component lies in that component.  A component containing no
center is therefore an isolated singleton.  Since $\Gamma$ is finite, so is
$J_C$.
\end{proof}

Define a poset $P_X$ as follows.  A condition is a pair
$p=(\Gamma_p,\varepsilon_p)$, where
\begin{itemize}
\item $\Gamma_p$ is a finite subset of $\cS(X)\times\N$;
\item $\varepsilon_p:X\rightharpoonup\{0,1\}$ is a finite partial map.
\end{itemize}
The set $\Gamma_p$ generates an equivalence relation $\sim_p$ on $X$: for
every $(s,N)\in\Gamma_p$ and every $n\geq N$, impose
$s_n\sim_p s_\infty$, and take the equivalence closure.  The condition $p$ is
admissible if
\begin{equation}\label{eq:admissible}
 ~~~~~~~~~~~~~~~~~~~~~~~~~~~~~~~~~ u\sim_p v,\
  u,v\in\operatorname{dom}(\varepsilon_p)
  \quad\Longrightarrow\quad
  \varepsilon_p(u)=\varepsilon_p(v).
\end{equation}
Only admissible pairs are retained. An equivalence class is called
\emph{colored} if it meets $\operatorname{dom}(\varepsilon_p)$; its color is
well-defined by \eqref{eq:admissible}.

Let $P_X$ be the set of all admissible pairs. For $p=(\Gamma_p, \varepsilon_p), q=(\Gamma_q, \varepsilon_q)\in P_X$, put $q\leq p$ if

$$  \Gamma_q\supseteq\Gamma_p
  \quad\text{and}\quad
  \varepsilon_q\supseteq\varepsilon_p.$$

\noindent Thus smaller conditions carry more information.

For $s\in\cS(X)$, let

$$  D_s=\{p\in P_X:(s,N)\in\Gamma_p\text{ for some }N\in\N\}.$$

\begin{lemma}[Tail-density lemma]\label{lem:tail-density}
For every $s\in\cS(X)$, the set $D_s$ is dense in $P_X$.
\end{lemma}

\begin{proof}
Fix $p\in P_X$, write $s_n\to x$, and let $C_0=[x]_{\sim_p}$.  Only
finitely many $\sim_p$-classes are colored, because each colored class meets
the finite set $\operatorname{dom}(\varepsilon_p)$.

Let $C\neq C_0$ be a colored class.  By
Lemma~\ref{lem:finite-component-structure}, either $C=\{z\}$ is an isolated
singleton, or it is the union of finitely many tail blocks
\[
  R(t^j,N_j)=\{t^j_\infty\}\cup\{t^j_n:n\geq N_j\}
\]
whose centers lie in $C$.  In the singleton case, $z\neq x$, so the sequence
$s$ takes the value $z$ only finitely often.  In the second case, no such
tail block contains $x$, since otherwise its center, and hence the whole
class $C$, would lie in $C_0$.  Lemma~\ref{lem:tail-block-avoidance} now
shows that $s$ enters each block only finitely often.  Thus

$$  \{n:s_n\in C\}$$

\noindent is finite in either case.

Taking the union over the finitely many colored classes other than $C_0$,
choose $N$ such that for every $n\geq N$, the old class
$[s_n]_{\sim_p}$ is either $C_0$ or uncolored.  After adjoining $(s,N)$, the
new edges all join $C_0$ to those old classes.  Hence the only old classes
that are merged are $C_0$ and uncolored classes; every other old component
is unchanged.  No two classes carrying different colors can therefore be
identified.  Thus

$$  q=(\Gamma_p\cup\{(s,N)\},\varepsilon_p)$$

\noindent is admissible, with $q\leq p$ and $q\in D_s$.
\end{proof}

\subsection{The regular-open completion}

We recall a concrete completion of an arbitrary poset. Give a poset $P$ the
lower Alexandrov topology, whose open sets are the lower sets. The regular
open subsets form a complete Boolean algebra $\RO(P)$; see, for example,
\cite{Sikorski-1969}. If $U$ and $V$ are regular open, then their meet is
$U\cap V$, the Boolean complement of $U$ is $\Int(P\setminus U)$, and the
join of a family $(U_i)_i$ is
$\operatorname{ro}(\bigcup_i U_i)$. For $A\subseteq P$, write

$$  \operatorname{ro}(A)=\Int\cl A,$$

\noindent and for $p\in P$ put

$$  \rho(p)=\operatorname{ro}(\downarrow p)\in\RO(P).$$

\noindent The direction of the Alexandrov topology can be checked directly.  The least
open neighbourhood of $u$ is ${\downarrow} u$, and therefore
\begin{equation}\label{eq:rho-characterization}
 ~~~~~~~~~~~~~~~~~~~~~~~~~~~~~~ \rho(p)=\{u\in P:\text{every }v\leq u\text{ is compatible with }p\}.
\end{equation}
Indeed, $u\in\cl({\downarrow} p)$ precisely when $u$ and $p$ have a common
lower bound, and taking the interior requires the same condition for every
$v\leq u$.  Formula~\eqref{eq:rho-characterization} also shows explicitly
that no separativity assumption on $P$ is being made.
In the following lemma, compatibility means the existence of a common lower
bound. A subset $D$ is \emph{dense below $p$} if, for every $r\leq p$,
there is $q\in D$ with $q\leq r$.

\begin{lemma}\label{lem:RO-completion}
Let $P$ be a poset and $\B=\RO(P)$.  Then:
\begin{enumerate}[\rm (i)]
\item $\rho(p)\neq0$ for every $p\in P$, and $q\leq p$ implies
      $\rho(q)\leq\rho(p)$;
\item $p$ and $q$ are compatible if and only if
      $\rho(p)\wedge\rho(q)\neq0$;
\item $\{\rho(p):p\in P\}$ is order-dense in $\B$;
\item if $D\subseteq P$ is dense, then
      $\bigvee_{q\in D}\rho(q)=1$;
\item if $D$ is dense below $p$, then $\rho(p)=\bigvee\{\rho(q):q\leq p,\ q\in D\}.$

\end{enumerate}
\end{lemma}

\begin{proof}
The principal lower set ${\downarrow} p$ is a nonempty open set contained in
$\rho(p)$, proving the first assertion; monotonicity is immediate.

If $r\leq p,q$, then
${\downarrow} r\subseteq\rho(p)\cap\rho(q)$, so the meet is nonzero.
Conversely, take $u\in\rho(p)\cap\rho(q)$.  Since
$u\in\cl({\downarrow} p)$, the least open neighbourhood ${\downarrow} u$ meets
${\downarrow} p$; choose $v\leq u,p$.  The set $\rho(q)$ is lower, so
$v\in\rho(q)\subseteq\cl({\downarrow} q)$.  Hence ${\downarrow} v$ meets
${\downarrow} q$, and some $r$ satisfies $r\leq v,q$.  Then $r\leq p,q$.
This proves (ii).

For (iii), let $0\neq U\in\B$ and choose $p\in U$.  Since $U$ is lower,
$\downarrow p\subseteq U$, and regularity gives

$$  \rho(p)=\Int\cl(\downarrow p)\subseteq\Int\cl U=U.$$

Let $D$ be dense and set $b=\bigvee_{q\in D}\rho(q)$.  If $b\neq1$, use
(iii) to choose $p$ with $0\neq\rho(p)\leq\neg b$.  Take $q\leq p$ in $D$.
Then $0\neq\rho(q)\leq\rho(p)\leq\neg b$, while $\rho(q)\leq b$, a
contradiction.  This proves (iv).

For (v), let the right-hand side be $c\leq\rho(p)$.  If
$c<\rho(p)$, choose $r$ with
$0\neq\rho(r)\leq\rho(p)\wedge\neg c$.  By (ii), $r$ and $p$ have a
common lower bound $r'$.  Density below $p$ gives $q\in D$ with $q\leq r'$.
Then $0\neq\rho(q)\leq\rho(r)\leq\neg c$, although $\rho(q)\leq c$ by the
definition of $c$.  This contradiction proves (v).
\end{proof}

\subsection{Boolean values attached to points}

Apply Lemma~\ref{lem:RO-completion} to $P_X$ and put

$$  \B_X=\RO(P_X).$$

\noindent For $x\in X$ and $i\in\{0,1\}$, let

$$  P(x,i)=\{p\in P_X:[x]_{\sim_p}\text{ is colored }i\}$$

\noindent and define

$$  b_x^i=\bigvee_{p\in P(x,i)}\rho(p)\in\B_X.$$

\begin{lemma}\label{lem:boolean-colors}
For every $x\in X$,

 $$ b_x^0\wedge b_x^1=0,
  ~
  b_x^0\vee b_x^1=1.$$

\noindent Consequently, if $b_x=b_x^1$, then $\neg b_x=b_x^0$.
\end{lemma}

\begin{proof}
Every condition has an extension deciding the color of $x$.  If the class of
$x$ is uncolored, assign either color to $x$; admissibility is preserved.
Thus $P(x,0)\cup P(x,1)$ is dense, and
Lemma~\ref{lem:RO-completion}(iv) gives the join equality.

A condition in $P(x,0)$ and a condition in $P(x,1)$ have no common lower
bound, since a common extension would assign two colors to one equivalence
class. Hence Lemma~\ref{lem:RO-completion}(ii) gives
$\rho(p)\wedge\rho(q)=0$ whenever $p\in P(x,0)$ and $q\in P(x,1)$.
Distributivity of finite meets over arbitrary joins in $\B_X$ now yields

$$  b_x^0\wedge b_x^1
  =\bigvee_{p\in P(x,0),\ q\in P(x,1)}
     (\rho(p)\wedge\rho(q))=0.$$
\end{proof}

\begin{lemma}\label{lem:point-separation}
If $x\neq y$ in $X$, then $b_x\neq b_y$.
\end{lemma}

\begin{proof}
Take the condition with no tail constraints and with the finite coloring
$\varepsilon(x)=1$, $\varepsilon(y)=0$.  It is admissible and satisfies

$$  0<\rho(p)\leq b_x\wedge\neg b_y.$$
\end{proof}

For $x,y\in X$, set

$$  E(x,y)=(b_x\wedge b_y)\vee(\neg b_x\wedge\neg b_y).$$

\begin{lemma}\label{lem:local-equality}
If $x\sim_p y$, then $\rho(p)\leq E(x,y)$.
\end{lemma}

\begin{proof}
Let

$$  D=\{q\leq p:q\in P(x,0)\cup P(x,1)\}.$$

\noindent The set $D$ is dense below $p$: if $r\leq p$ does not yet color the
$\sim_r$-class of $x$, one may color that class without changing
$\Gamma_r$. Since $x\sim_p y$ and extensions preserve all relations already
present in $p$, every $q\in D$ colors $x$ and $y$ alike. Consequently,

$$  \rho(q)\leq b_x\wedge b_y
  \quad\text{or}\quad
  \rho(q)\leq\neg b_x\wedge\neg b_y.$$

\noindent Lemma~\ref{lem:RO-completion}(v) therefore gives

$$  \rho(p)=\bigvee_{q\in D}\rho(q)\leq E(x,y).$$
\end{proof}

\begin{proposition}\label{prop:sequence-to-order}
If $s_n\to s_\infty$ in $X$, then

$$  \bigvee_{N\in\N}\bigwedge_{n\geq N}E(s_n,s_\infty)=1.$$

\noindent Equivalently, with $b_n=b_{s_n}$ and $b=b_{s_\infty}$,

$$  \bigwedge_{N\in\N}\bigvee_{n\geq N}(b_n\mathbin{\triangle}b)=0,$$

\noindent where $\triangle$ denotes Boolean symmetric difference.
\end{proposition}

\begin{proof}
For $p\in D_s$, choose $N$ with $(s,N)\in\Gamma_p$.  Then
$s_n\sim_p s_\infty$ for every $n\geq N$, and
Lemma~\ref{lem:local-equality} yields

$$  \rho(p)\leq\bigwedge_{n\geq N}E(s_n,s_\infty).$$

\noindent The density of $D_s$, Lemma~\ref{lem:tail-density}, and
Lemma~\ref{lem:RO-completion}(iv) prove the first equality.  The second is
its Boolean complement.
\end{proof}

We record the order-convergence calculation in full, since it is the step
which turns finite tail constraints into Scott convergence.

\begin{lemma}\label{lem:order-convergence}
Let $b_n,b$ belong to a complete Boolean algebra and suppose that

$$  \bigwedge_N\bigvee_{n\geq N}(b_n\mathbin{\triangle}b)=0.$$

\noindent Put

$$  a_N=\bigwedge_{n\geq N}b_n,
  ~
  c_N=\bigwedge_{n\geq N}\neg b_n.$$

\noindent Then

$$  \bigvee_N a_N=b,
  ~
  \bigvee_N c_N=\neg b.$$
\end{lemma}

\begin{proof}
Let
$r_N=\bigvee_{n\geq N}(b_n\mathbin{\triangle}b)$.  Then
$r_N$ decreases and $\bigwedge_Nr_N=0$.  For $n\geq N$, the elements
$b_n$ and $b$ agree on $\neg r_N$.  Hence

$$  b\wedge\neg r_N\leq a_N,
  \qquad
  \neg b\wedge\neg r_N\leq c_N.$$

\noindent Since $\bigvee_N\neg r_N=1$, it follows that
$b\leq\bigvee_Na_N$ and $\neg b\leq\bigvee_Nc_N$.

For the reverse inequality, fix $N$.  For every $n\geq N$,

$$  a_N\wedge\neg b\leq b_n\wedge\neg b
                    \leq b_n\mathbin{\triangle}b.$$

\noindent Thus $a_N\wedge\neg b\leq\bigwedge_{n\geq N}(b_n\mathbin{\triangle}b)$.
For any $M$, choose $k\geq\max\{M,N\}$.  Then

$$  \bigwedge_{n\geq N}(b_n\mathbin{\triangle}b)
     \leq b_k\mathbin{\triangle}b\leq r_M.$$

\noindent Taking the meet over $M$ gives $a_N\wedge\neg b=0$, so $a_N\leq b$.
The same argument, with complements interchanged, gives $c_N\leq\neg b$.
Taking joins over $N$ completes the proof.
\end{proof}

\begin{theorem}[Boolean realization theorem]\label{thm:realization}
Let $X$ be a US-space.  There are a complete Boolean algebra $\B_X$ and an
injective map

$$  j:X\longrightarrow G_{\B_X},
  ~ j(x)=(b_x,\neg b_x),$$

\noindent which preserves convergent sequences.  If $X$ is sequential, then $j$ is
continuous.  If $X$ is both sequential and compact, then $j$ is a
topological embedding.
\end{theorem}

\begin{proof}
Injectivity follows from Lemma~\ref{lem:point-separation}.  Suppose that
$s_n\to s_\infty$, and write $b_n=b_{s_n}$ and $b=b_{s_\infty}$.  Define
$a_N,c_N$ as in Lemma~\ref{lem:order-convergence} and put
$d_N=(a_N,c_N)\in\B_X^2$.  By Proposition~\ref{prop:sequence-to-order} and
Lemma~\ref{lem:order-convergence},

$$  d_N\uparrow(b,\neg b)=j(s_\infty).$$

\noindent Moreover, $d_N\leq(b_n,\neg b_n)=j(s_n)$ whenever $n\geq N$.

Let $O$ be a neighbourhood of $j(s_\infty)$ in the relative Scott topology
on $G_{\B_X}$.  Write $O=U\cap G_{\B_X}$ with $U$ Scott open in
$\B_X^2$.  Since $\bigvee_Nd_N=j(s_\infty)\in U$, some $d_N$ belongs to
$U$.  The upperness of $U$ then gives $j(s_n)\in U$ for all $n\geq N$.
The approximants $d_N$ need not themselves lie in the complement graph:
they are used only as a directed family in the ambient Scott space
$\Sigma(\B_X^2)$.  Hence $j(s_n)\to j(s_\infty)$ in the relative Scott
topology.

If $X$ is sequential, a map preserving convergent sequences is continuous:
the inverse image of every open set is sequentially open and therefore open.
Assume in addition that $X$ is compact.  By
Proposition~\ref{prop:graph-kc}, $G_{\B_X}$ is \emph{KC}.  For every closed
$F\subseteq X$, the set $F$ is compact, so $j(F)$ is compact and therefore
closed in $G_{\B_X}$.  Thus the continuous bijection
$j:X\to j(X)$ is closed and hence a homeomorphism.
\end{proof}

\section{The counterexample}\label{sec:main}

Van Douwen constructed a topology on $\N$ which is compact,
Fr\'echet and anti-Hausdorff, while convergent sequences have unique
limits \cite[pp.~149--151]{Douwen-1993}.  The same package of properties is
restated explicitly in \cite[Theorem~5.2]{Xu-Xu-Ji-Qiu-2026}.  We use the
following consequence.

\begin{theorem} (Douwen Theorem \cite{Douwen-1993})\label{thm:van-douwen}
There exists a non-singleton countable compact Fr\'echet US-space
$X$ which is anti-Hausdorff.
\end{theorem}

\begin{theorem}\label{thm:main}
There exists a complete Boolean algebra $C$ whose Scott space is not
co-sober.
\end{theorem}

\begin{proof}
Let $X$ be the space in Theorem~\ref{thm:van-douwen}.  It is sequential, so
Theorem~\ref{thm:realization} provides a complete Boolean algebra $B$ and a
topological embedding $j:X\to G_B$.  Put $M=j(X)$.  Then $M$ is
non-singleton, compact and irreducible.  The complete Boolean algebra
$C=B^2$ satisfies the hypothesis of Theorem~\ref{thm:reduction}, and hence
$\Sigma C$ is not co-sober.
\end{proof}

\begin{corollary}\label{cor:stronger}
There exists a complete Boolean algebra whose Scott space is well-filtered
but neither sober nor co-sober.
\end{corollary}

\begin{proof}
For the algebra $C$ in Theorem~\ref{thm:main}, non-sobriety follows from
Proposition~\ref{prop:nonsober-reduction}.  Every complete lattice has a
well-filtered Scott space by \cite{Xi-Lawson-2017}.
\end{proof}

The theorem also answers the unrestricted existence parts of the analogous
questions for complete lattices and complete Heyting algebras.  It does not,
however, produce a countable example in those larger classes.  For Boolean
algebras, countability is impossible for a different reason.

\begin{proposition}\label{prop:countable-finite}
Every countable complete Boolean algebra is finite.  In particular, the
Scott space of every countable complete Boolean algebra is co-sober.
\end{proposition}

\begin{proof}
Let $B$ be an infinite complete Boolean algebra.  We first find pairwise
disjoint nonzero elements $a_0,a_1,\ldots$.  If $B$ has infinitely many
atoms, choose countably many of them. Otherwise let $a$ be the join of all
atoms. If $a=1$, every element of $B$ is the join of the atoms below it; with
only finitely many atoms this makes $B$ finite, contrary to the assumption.
Hence $c_0=\neg a$ is nonzero, and the interval below $c_0$ is atomless.  Recursively choose
$0<c_{n+1}<c_n$ and put

$$  a_n=c_n\wedge\neg c_{n+1}.$$

\noindent Then each $a_n$ is nonzero and the family is pairwise disjoint.

For $S\subseteq\N$, set $b_S=\bigvee_{n\in S}a_n$.  Distinct subsets give
distinct elements, so $\mathcal{P}(\N)$ embeds into $B$.  Hence
$|B|\geq2^{\aleph_0}$, and a countable complete Boolean algebra must be
finite.

Now let $B$ be finite and $K$ a $k$-irreducible compact saturated subset of
$\Sigma B$.  If
$\Min K=\{m_1,\ldots,m_r\}$, then

$$  K={\uparrow} m_1\cup\cdots\cup{\uparrow} m_r.$$

\noindent Repeated use of $k$-irreducibility gives $K={\uparrow} m_i$ for some $i$.
Thus $\Sigma B$ is co-sober.
\end{proof}

\begin{remark}\label{rem:cardinality}
For the countable space $X$ in Theorem~\ref{thm:van-douwen}, the set of all
convergent sequences has cardinal at most
$\mathfrak c=2^{\aleph_0}$.  Hence $|P_X|\leq\mathfrak c$ and
$|\RO(P_X)|\leq2^{\mathfrak c}$.  Proposition~\ref{prop:countable-finite}
shows that every infinite complete Boolean algebra has cardinal at least
$\mathfrak c$.  The present construction therefore gives

$$  \mathfrak c\leq|C|\leq2^{\mathfrak c}.$$

It remains natural to ask whether a counterexample can be chosen of cardinal
exactly $\mathfrak c$, or subject to a chain condition or a separability
requirement.
\end{remark}

\section{Conclusion}

We have constructed a complete Boolean algebra whose Scott space is not
co-sober. The construction first realizes a compact sequential \emph{US}-space in a
Boolean complement graph and then saturates an embedded copy of van
Douwen's compact anti-Hausdorff space. The \emph{KC} property of the complement
graph transfers irreducibility to the saturation. The same embedded space
also yields a non-point irreducible Scott-closed set, so the resulting Scott
space is non-sober while remaining well-filtered.

The realization theorem is independent of the counterexample: every compact
sequential \emph{US}-space admits a complement-graph realization. Natural questions
remain about the size and structure of the Boolean algebra. In particular,
one may ask whether a counterexample can be chosen of cardinality
$\mathfrak c$, or with a chain condition or a separability property.

%\vspace{0.5cm}

%\noindent{\bf References}


\begin{thebibliography}{99}
\bibitem{Abramsky-Jung-1994} S. Abramsky, A. Jung, Domain theory, in: Semantic Structures, Handbook of Logic in Computer Science, vol.3, Clarendon Press, 1994, pp.1-168.

\bibitem{Douwen-1993} E. Douwen, An anti-Hausdorff Fr\'echet space in which convergent sequences have
unique limits, Topol. Appl. 51 (1993) 147-158.

\bibitem{Drake-Thron-1965} D. Drake, W. Thron, On the representation of an abstract
lattice as the family of closed sets of a topological space, Trans. Amer. Math. Soc. 120(1965) 57-71.

\bibitem{Engelking-1989} R. Engelking, General Topology, Polish Scientific Publishers, Warzawa, 1989.

\bibitem{Escardo-Lawson-Simpson-2004} M. Escard\'{o}, J. Lawson, A. Simpson, Comparing Cartesian closed categories of (core) compactly generated spaces, Topol.Appl. 143(2004) 105-145.

\bibitem{GHKLMS-2003}G. Gierz, K. Hofmann, K. Keimel, J.D. Lawson, M. Mislove, D. Scott, Continuous Lattices and Domains, Encyclopedia of Mathematics and its Applications 93. Cambridge University Press, 2003.

\bibitem{Goubault-2013} J. Goubault-Larrecq, Non-Hausdorff Topology and Domain Theory, New Mathematical Monographs, vol. 22, Cambridge University Press, 2013.

%\bibitem{Heckmann-1992} R. Heckmann, An upper power domain construction in terms of strongly compact sets, in: Lecture Notes in Computer Science, vol. 598, Springer-Verlag, New York, 1992, pp. 272-293.

%\bibitem{Heckmann-Keimel-2013} R. Heckmann, K. Keimel, Quasicontinuous domains and the Smyth powerdomain, Electron. Notes Theor. Comput. 298 (2013) 215-232.


%\bibitem{Hochster-1969} M. Hochster, Prime ideal structure in commutative rings, Trans. Amer. Math. Soc. 142 (1969) 43-60.

\bibitem{Hoffmann-1979-1} R. Hoffmann, On the sobrification remainder $^S\!X-X$, Pac. J. Math. 83 (1979) 145-156.

%\bibitem{Isbell-1982} J. Isbell, Completion of a construction of Johnstone, Proceedings of the American Mathematical Society, 85 (1982) 333-334.

\bibitem{Johnstone-1981} P. Johnstone, Scott is not always sober, in: Continuous Lattices, Lecture Notes in Math., vol. 871, Springer-Verlag, 1981, pp. 282-283.

\bibitem{Kunzi-Zypen-2004} H. K\"unzi, D. Zypen, Maximal (sequentially) compact topologies, Appl. Categ. Struct. 12 (2004) 173-187.

%\bibitem{Lawson-Xu-2024-2} J. Lawson, X. Xu, $T_0$-spaces and the lower topology, Math. Struct. Comput. Sci. 34 (2024) 467-490.

%\bibitem{Schalk-1993} A. Schalk, Algebras for Generalized Power Constructions. PhD Thesis, Technische Hochschule Darmstadt, 1993.

\bibitem{Shen-Wu-Xi-Zhao-2020} C. Shen, G. Wu, X. Xi, D. Zhao, Sober Scott space is not always co-sober,  Topol. Appl. 282 (2020) 107316.


\bibitem{Sikorski-1969} R. Sikorski, Boolean Algebras, 3rd Edition, Vol. 25 of Ergebnisse der Mathematik und ihrer Grenzgebiete, Springer-Verlag, Berlin, 1969.


%\bibitem{Thron-1962} W. Thron, Lattice-equivalence of topological spaces, Duke Math. J. 29 (1962) 671-679.

\bibitem{Wen-Xu-2018} X. Wen, X. Xu, Sober is not always co-sober, Topol. Appl. 250 (2018) 48-52.

\bibitem{Wilansky-1967} A. Wilansky, Between $T_1$ and $T_2$, Amer. Math. Monthly 74 (1967) 261-266.

\bibitem{Xi-Shen-Zhao-2026} X. Xi, C. Shen, D. Zhao, The complete Boolean algebra of regular open sets in real line is not sober, 2026, arXiv:2608.00408.

\bibitem{Xi-Lawson-2017} X. Xi and J. Lawson, On well-filtered spaces and ordered sets, Topol. Appl. 228 (2017) 139-144.


%\bibitem{Xu-2016} X. Xu, Order and Topology, Beijing: Science Press, 2016.


\bibitem{Xu-2026} X. Xu, Some open problems on co-sober spaces, Invited talk at ISDT 2026, Chengdu,4--8 August 2026.

%\bibitem{Xu-Shen-Xi-Zhao-2020-2} X. Xu, C. Shen, X. Xi, D. Zhao, On $T_0$ spaces determined by well-filtered spaces, Topol. Appl. 282 (2020) 107323.




\bibitem{Xu-Xu-Ji-Qiu-2026} X. Xu, F. Xu, H. Ji, D. Qiu, Smyth power space of a co-sober space is not always co-sober, 2026, arXiv:2608.02047.



\bibitem{Xu-Zhao-2021} X. Xu, D. Zhao, Some open problems on well-filtered spaces and sober spaces,
    Topol. Appl. 301 (2021) 107540.

\bibitem{Zhao-He-Wang-2023}B. Zhao, Z. He, K. Wang, On co-sober spaces, Houston J. Math. 49(4) (2023) 971-988.

%\bibitem{Zhao-Ho-2015}D. S. Zhao, W. K. Ho, On topologies defined by irreducible sets, Journal of Logical and Algebraic Methods in Programming, 84 (2015) 185-195.

\end{thebibliography}
\end{document}